\documentclass[12pt]{amsart}
\usepackage{bbding}
\usepackage[margin=2cm]{geometry}

\usepackage{amsfonts}
\usepackage{graphicx}

\newcommand{\R}{{\mathbf R}}

\newtheorem{theorem}{Theorem}[section]
\newtheorem{lemma}[theorem]{Lemma}
\newtheorem{corollary}[theorem]{Corollary}

\def\proofbox{\mbox{\bf Q.E.D.}\\}

\newcommand{\Kn}{\mathcal{K}^n}

\newcommand{\V}{\mathrm{V}}

\title{The cone volume measure of antipodal points}
\author{K\'aroly J. B\"or\"oczky}
\address{Alfr\'ed R\'enyi Institute of Mathematics, Hungarian Academy
  of Sciences, Reltanoda u. 13-15, H-1053 Budapest, Hungary, and
Department of Mathematics, Central European University, Nador u 9, H-1051, Budapest, Hungary}
\author{P\'al Heged\H{u}s}
\address{Department of Mathematics, Central European University, Nador u 9, H-1051, Budapest, Hungary}
\email{}

\begin{document}
\maketitle

\begin{abstract}
The optimal condition of the cone volume measure of a pair of antipodal points is proved and analyzed.
\end{abstract}

\section{Introduction}

Let $\Kn$ be the set of all convex bodies in $\R^n$ having non-empty
interiors, 
i.e., $K\in \Kn$ is a convex compact subset of the
$n$-dimensional Euclidean space $\R^n$ with ${\rm int}\,K\ne\emptyset$. As
usual, 
we denote by $\langle\cdot,\cdot\rangle$ the inner product on
$\R^n\times\R^n$ with  associated Euclidean norm
$\|\cdot\|$. $S^{n-1}\subset\R^n$ denotes the $(n-1)$-dimensional unit
sphere, i.e., $S^{n-1}=\{x\in\R^n : \|x\|=1\}$.

For $K\in\Kn$, we write $S_K(\cdot)$ and $h_K(\cdot)$ to denote 
its surface area measure and support function, respectively,  
and $\nu_K$ to denote the Gau{\ss} map assigning  the exterior unit
normal $\nu_K(x)$ 
to an $x\in \partial_*K$, where $\partial_*K$  consists of all points in the
boundary $\partial K$ of $K$ having an unique outer normal
vector. We write $F(K,u)$ to denote the face of $K$ with exterior normal $u\in S^{n-1}$; namely,
$$
F(K,u)=\{x\in K:\,\langle x,u\rangle=h_K(u)\}.
$$
 If the origin $o$ lies in the interior of $K\in\Kn$,  then
the {\em cone volume measure of $K$} on $S^{n-1}$ is given by 
\begin{equation} 
\label{Gauss-cone}
\V_K(\omega)=\int_{\omega}\frac{h_K(u)}n\,dS_K(u)=
\int_{\nu_K^{-1}(\omega)}\frac{\langle x,\nu_K(x)\rangle}n\,d{\mathcal H}_{n-1}(x),
\end{equation}
where  $\omega\subset S^{n-1}$ is a Borel set and, in general,
${\mathcal H}_{k}(x)$ denotes the $k$-dimensional
Hausdorff-measure. Instead of ${\mathcal H}_n(\cdot)$, we also 
write $\V(\cdot)$ for the $n$-dimensional volume.  For general references regarding convex bodies see, e.g., 
P.M. Gruber \cite{Gruberbook}, R. Schneider \cite{Sch93} and
A.C. Thompson \cite{Tho96}.

The name cone volume measure stems from the fact that if $K$ is a polytope with facets $F_1,\ldots,F_m$ and corresponding exterior unit normals $u_1,\ldots,u_m$, then
$$
\V_K(\omega)=\sum_{i=1}^m\V([o,F_i])\delta_{u_i} (\omega).
$$
Here $\delta_u$ is the Dirac delta measure  on
$S^{n-1}$ at $u\in S^{n-1}$, and for $x_1,\dots,x_m\in\R^n$ and 
subsets $S_1,\dots,S_L\subseteq \R^n$ we denote the convex hull of the
set $\{x_1,\dots,x_m, S_1,\dots,S_l\}$ by  
$[x_1,\dots,x_m,S_1,\dots,S_l]$.  With this notation $[o,F_i]$ is the cone with apex $o$ and
basis $F_i$.

In recent years, cone volume measures have 
appeared and were studied in various contexts, see, e.g.,  
F. Barthe, O. Guedon, S. Mendelson and  A. Naor \cite{BGMN05}, 
K.J. B\"or\"oczky, E. Lutwak, D. Yang and G. Zhang \cite{BLYZ12,
  BLYZ13}, 
M. Gromov and V.D. Milman \cite{GrM87}, 
M. Ludwig \cite{Lud10}, 
M. Ludwig and  M. Reitzner \cite{LuR10}, 
E. Lutwak, D. Yang and G. Zhang \cite{LYZ05}, 
A. Naor \cite{Nao07}, 
A. Naor and D. Romik \cite{NaR03}, 
G.  Paouris and E. Werner \cite{PaW12}, 
A. Stancu \cite{Sta12}, 
G.~Zhu \cite{Zhu14a, Zhu14b, Zhu14c}.

In particular, cone volume measure are the subject of 
the  {\em logarithmic Minkowski problem}, which is the 
particular interesting limiting case $p=0$ of  the  general $L_p$-Minkowski problem -- one
of the central problems in  convex geometric analysis (see  E. Lutwak \cite{Lut93b}). It is the task: 

\smallskip \noindent 
{\em Find necessary and sufficient conditions for a Borel measure
  $\mu$ on $S^{n-1}$ to be the cone volume measure $\V_K$ of 
  $K\in\Kn$ (with $o$ in its interior).}

When $\mu$ has a density $f$, with respect to spherical Lebesgue measure, the logarithmic Minkowski problem involves establishing existence for the
Monge-Amp\`{e}re type equation:
$$
h\det(h_{ij}+h\delta_{ij})=f,
$$
where $h_{ij}$ is the covariant derivative of $h$ with respect to an orthonormal frame on $S^{n-1}$ and $\delta_{ij}$ is the Kronecker delta.

In the  recent  paper \cite{BLYZ13}, 
K.J. B\"or\"oczky, E. Lutwak, D. Yang and G. Zhang characterize the cone volume
measures of origin-symmetric convex bodies. In order to state their
result we say that a Borel measure $\mu$ on $S^{n-1}$  satisfies
the {\em subspace concentration condition  
for a  linear subspace $L\subset \R^n$}, if
\begin{equation}
\label{scc}
\mu(L\cap S^{n-1})\leq \frac{\dim L}{n}\,\mu(S^{n-1}),
\end{equation}
and equality in (\ref{scc}) implies the existence of a complementary linear subspace $\widetilde{L}$ such that
\begin{equation}
\mu(\widetilde{L}\cap S^{n-1})=\frac{\dim\widetilde{L}}{n}\,\mu(S^{n-1}),
\label{eq:sccequality}
\end{equation}
and hence ${\rm supp}\,\mu\subset L\cup \widetilde{L}$, i.e., the
support of the measure ``lives'' in  $L\cup \widetilde{L}$. In addition, $\mu$
 satisfies
the {\em subspace concentration condition} if it satisfies the subspace concentration condition
for any  linear subspace.

Via the subspace concentration condition, the logarithmic Minkowski
problem was settled in \cite{BLYZ13}  in the
symmetric case.\\
  
\noindent\textbf{Theorem A.} \emph{A non-zero finite even Borel measure on the
$S^{n-1}$ is the cone-volume measure of an
origin-symmetric convex body in $\mathbb{R}^{n}$ if and only if it
satisfies the subspace concentration condition.}\\

This result  was proved earlier for discrete measures  on $S^1$, i.e.,
for polygons, by A. Stancu \cite{Sta02, Sta03}.  
For cone-volume measures of origin-symmetric polytopes, the necessity
of \eqref{scc} was independently shown by   M. Henk, A. Sch\"urmann
and J.M.Wills \cite{HSW05} and 
B. He, G. Leng and K. Li \cite{HLL06}.

Theorem~A shows that the subspace concentration condition is a natural condition for all even measures that may arise as the cone-volume measures of origin-symmetric convex bodies. Actually, B\"or\"oczky, Henk \cite{BH} proved that the cone volume measure of any convex body whose centroid is the origin satisfies the subspace concentration condition.

Next, we say that a linear subspace $L$ of  $\R^n$ is {\emph essential} with respect to a Borel measure $\mu$ on $S^{n-1}$ if $L\cap{\rm supp}\mu$ is not concentrated on any closed hemisphere of $L\cap S^{n-1}$. In other words,
$L\cap{\rm supp}\mu$ contains $1+{\rm dim}L$ vectors spanning $L$ such that the origin is a positive linear combination of these vectors. As a generalization of the discrete case of Theorem~A and the main result of G. Zhu \cite{Zhu14a}, the following is 
proved  in the beautiful paper A. Stancu \cite{Sta02} if  $n=2$, 
and in K.J. B\"or\"oczky, P. Heged\H{u}s, G. Zhu \cite{BHZ} if $n\geq 3$.\\

\noindent\textbf{Theorem B.} \emph{If $\mu$ is a discrete measure on $S^{n-1}$, $n\geq 2$, that is not concentrated on any closed hemisphere, and $\mu$ satisfies the subspace concentration condition with respect to any essential linear subspace $L$, then $\mu$ is the cone-volume measure of a polytope in $\mathbb{R}^{n}$.}\\

Interestingly enough, even if the methods of the papers \cite{BHZ} and \cite{Sta02} are quite different, both papers
 need the condition that the subspace concentration condition holds with respect to any essential linear subspace exactly for the same reason; namely, to ensure that the related extremal problem  has bounded solution.

We do not even have a conjecture on what properties may characterize cone volume measures.  Actually, having  subspace concentration condition holds with respect to essential linear subspaces is not a necessary condition for a cone volume measure. As an example, let $u_1,\ldots,u_n$ be an orthonormal basis of $\R^n$, and let
$W=\{x\in u_1^\bot:\,|\langle x,u_i\rangle|\leq 1,\;i=2,\ldots,n\}$ be an $(n-1)$-dimensional
cube. For $r>0$ and $i=1,\ldots,n-1$, $L_i={\rm lin}\{u_1,\ldots,u_i\}$ is an essential subspace for
the cone volume measure of the truncated pyramid $P_r=[-ru_1-rW,u_1+W]$. If $r>0$ is small, then $P_r$ approximates
$[o,u_1+W]$, and
$$
V_{P_r}(L_i\cap S^{n-1})>V_{P_r}(\{u_1\})=V([o,u_1+W])>\mbox{$\frac{i}{n}$}\, V(P_r).
$$

As a modest first step towards understanding cone volume measures of general convex bodies, the goal of this noteis to  characterize cone volume measure of a pair of antipodal points. For $\alpha,\beta>0$, we consider the auxiliary function
\begin{eqnarray}
\nonumber
\varphi(\alpha,\beta,n)&=&\min_{\varrho>0}
\left(\frac{\alpha}{\varrho^{n-1}}+\beta \varrho^{n-1}\right)\sum_{i=0}^{n-1}\varrho^{n-1-2i}\\
\label{phidef}
&=& \alpha+\beta+\min_{\varrho>0}\sum_{i=1}^{n-1}\left(\alpha\varrho^{-2i}+\beta\varrho^{2i}\right).
\end{eqnarray}
As we will see in Section~\ref{sec2}, the minimum is attained at a unique $\varrho=\varrho_0(\alpha/\beta,n)>0$ 
(depending only on $\alpha/\beta$ and $n$).

\begin{theorem}
\label{oppositefacets}
If $K$ is a convex body containing the origin in its interior with $V(K)=1$,  and $V_K(\{u\})=\alpha>0$ and
 $V_K(\{-u\})=\beta>0$ for $u\in S^{n-1}$, then $\varphi(\alpha,\beta,n)\leq 1$, with equality if and only if
$F(K,u)$ and $F(K,-u)$ are homothetic, $K=[F(K,u), F(K,-u)]$, and
$\frac{h_K(u)}{h_K(-u)}=\frac{\alpha}{\beta}\cdot\varrho_0(\frac{\alpha}{\beta},n)^{-2(n-1)}$.
\end{theorem}

 For the planar case, we have the optimal condition on restrictions of a cone volume measure on $S^1$ to a pair of antipodal vectors.

\begin{corollary}
\label{oppositefacets2}
If $K$ is a convex body containing the origin in its interior in $\R^2$, and $u\in S^1$, then
$$
\sqrt{V_K(\{u\})}+\sqrt{V_K(\{-u\})}\leq \sqrt{V(K)},
$$
with equality if and only if $K$ is a trapezoid with two sides parallel to $u^\bot$, and $u^\bot$ contains the intersection of  the diagonals.
\end{corollary}

If $n\geq 3$, then we do not have a nice simple formula for  $\varphi(\alpha,\beta,n)$ and
$\varrho_0(\alpha/\beta,n)$, only estimates.  For
$\alpha,\beta>0$, we readily have
\begin{align}
\varphi(\alpha,\beta,n)=\varphi(\beta,\alpha,n)&
\mbox{ \ \ and \ \ $\varrho_0(\frac{\alpha}{\beta},n)=\varrho_0(\frac{\beta}{\alpha},n)^{-1}$,}\\
\varphi(\alpha,\alpha,n)=2n\alpha&
\mbox{ \ \ and \ \ $\varrho_0(1,n)=1$.}
\end{align}
Therefore we may assume that $\alpha>\beta$.

\begin{theorem}
\label{rho0}
If $n\geq 3$ and $\gamma>1$, then
$$
\gamma^{\frac3{8n-4}}<\varrho_0(\gamma,n)<\gamma^{\frac1{2n}},
$$
where $\lim_{\gamma\to 1^+}\frac{\log \varrho_0(\gamma,n)}{\log \gamma}=\frac3{8n-4}$ and
$\lim_{\gamma\to \infty}\frac{\log \varrho_0(\gamma,n)}{\log \gamma}=\frac1{2n}$.
\end{theorem}

For $\varphi(\alpha,\beta,n)$, we have the following bounds.

\begin{corollary}
\label{phinbig}
Let $n\geq 3$ and $\alpha\geq \beta>0$.
\begin{enumerate}
\item[(i)] $\displaystyle  \alpha+\beta+2(n-1)\sqrt{\alpha\beta}\leq\varphi(\alpha,\beta,n)
<\alpha+\beta+2(n-1)\alpha^{\frac{n-1}n}\beta^{\frac1n}.$
\item[(ii)] There exists $\varepsilon_0>0$ depending on $n$ such that if
 $\frac{\alpha}{\beta}=1+\varepsilon$ for $\varepsilon\in (0,\varepsilon_0)$, then
$$
\varphi(\alpha,\beta,n)=\alpha+\beta+2(n-1)\sqrt{\alpha\beta}+\frac{(n-1)(n-2)}{4(2n-1)}
\sqrt{\alpha\beta}\,\varepsilon^2+O(\sqrt{\alpha\beta}\,\varepsilon^3).
$$
\item[(iii)]  There exists $\gamma_0>0$ depending on $n$ such that if $\frac{\alpha}{\beta}>\gamma_0$, then
$$
\varphi(\alpha,\beta,n)=
\alpha+\beta+\frac{ n}{(n-1)^{\frac{n-1}n}}\,\alpha^{\frac{n-1}n}\beta^{\frac1n}+O\left(\alpha^{\frac{n-2}n}\beta^{\frac2n}\right).
$$
\end{enumerate}
\end{corollary}

We note that substituting $n=2$ into the bounds of Theorem~\ref{rho0} and Corollary~\ref{phinbig}, we obtain the formulas in
(\ref{phi2}).

Unfortunately, even Corollary~\ref{oppositefacets2} does not characterize cone volume measure in the plane
(see Lemma~\ref{nopolygon}). It is an intriguing problem to characterize at least the cone volume measure of polygons.

\section{Basic properties and the planar case}
\label{sec2}

Let $V_K$ be the cone volume measure on $S^{n-1}$ for a convex body $K$ in $\R^n$,  $n\geq 2$, with $o\in{\rm int}K$ and $V(K)=1$. We may write minimum in the definition (\ref{phidef}) of $\varphi(\alpha,\beta,n)$ because for fixed $\alpha,\beta,n>0$, the function
\begin{equation}
\label{fdef}
f_{\alpha,\beta,n}(\varrho)=\alpha+\beta+\sum_{i=1}^{n-1}\left(\alpha\varrho^{-2i}+\beta\varrho^{2i}\right)
\end{equation}
is strictly convex for $\varrho>0$, and tends to infinity as $\varrho$ tends to zero or infinity. In particular, there exists a unique
$\varrho_0(\alpha/\beta,n)>0$ (depending only on $\alpha/\beta$ and $n$) where $f_{\alpha,\beta,n}(\varrho)$ attains its minimum.
It follows from the inequality between the arithmetic and geometric mean that
\begin{equation}
\label{philow}
\varphi(\alpha,\beta,n)\geq \alpha+\beta+2(n-1)\sqrt{\alpha\beta},
\end{equation}
and if $n=2$, then
\begin{equation}
\label{phi2}
\varphi(\alpha,\beta,2)=(\sqrt{\alpha}+\sqrt{\beta})^2\mbox{ and }\varrho_0(\gamma,2)=\sqrt[4]{\gamma}.
\end{equation}

\noindent{\bf Proof of Theorem~\ref{oppositefacets}: } We write $|\cdot|$ to denote $(n-1)$-dimensional measure. After a volume preserving linear transform keeping ${\rm lin}u$ and $u^\bot$ invariant, we may assume that
\begin{equation}
\label{Fuproduct}
|F(K,u)|\cdot |F(K,-u)|=1.
\end{equation}
 Let $h_K(u)=a$ and
$h_K(-u)=b$, and for $t\in [0,1]$, let
$$
K_t=\left((at-(1-t)b)+u^\bot\right)\cap K.
$$
In particular, the
 Brunn-Minkowski inequality yields that
\begin{equation}
\label{Ktest}
|K_t|\geq \left(t|F(K,u)|^{\frac1{n-1}}+(1-t)|F(K,-u)|^{\frac1{n-1}}\right)^{n-1},
\end{equation}
with equality if and only if $K_t=tF(K,u)+(1-t)F(K,-u)$, and $F(K,u)$ and $F(K,-u)$ are homothetic. We deduce from the Fubini theorem that
\begin{eqnarray*}
V(K)&\geq &(a+b)\int_0^1 \left(t|F(K,u)|^{\frac1{n-1}}+(1-t)|F(K,-u)|^{\frac1{n-1}}\right)^{n-1}\,dt\\
&=&(a+b)\sum_{i=0}^{n-1}|F(K,u)|^{\frac{i}{n-1}}|F(K,-u)|^{\frac{n-1-i}{n-1}}{n-1\choose i}
\int_0^1 t^i(1-t)^{n-1-i}\,dt\\
&=&\frac{a+b}n\sum_{i=0}^{n-1}|F(K,u)|^{\frac{i}{n-1}}|F(K,-u)|^{\frac{n-1-i}{n-1}}.
\end{eqnarray*}
Let $\varrho=|F(K,u)|^{\frac1{n-1}}$, and hence $|F(K,-u)|=\varrho^{-(n-1)}$ by (\ref{Fuproduct}), and
$\alpha=a\varrho^{n-1}/n$ and  $\beta=b\varrho^{-(n-1)}/n$. Therefore
$$
1=V(K)\geq f_{\alpha,\beta,n}(\varrho)\geq \varphi(\alpha,\beta,n),
$$
where the equality conditions follow from the equality conditions for (\ref{Ktest}).
\proofbox

We observe that Corollary~\ref{oppositefacets2} is a consequence of  Theorem~\ref{oppositefacets} and (\ref{phi2}).

Corollary~\ref{oppositefacets2} readily characterizes  cone volume measures of quadrilaterals.
To show that even Corollary~\ref{oppositefacets2} does not characterize cone volume measure in the plane, we consider pentagons.
Choose $u_1,\ldots,u_5\in S^1$ in this order in a way such that  $u_1=-u_4$, $\langle u_1,u_2\rangle=\frac1{\sqrt{2}}$ and
$\langle u_4,u_i\rangle=\frac2{\sqrt{5}}$ for $i=3,5$ (see Figure 1). 
\begin{figure}[h]
\centering
\includegraphics[width=7 cm]{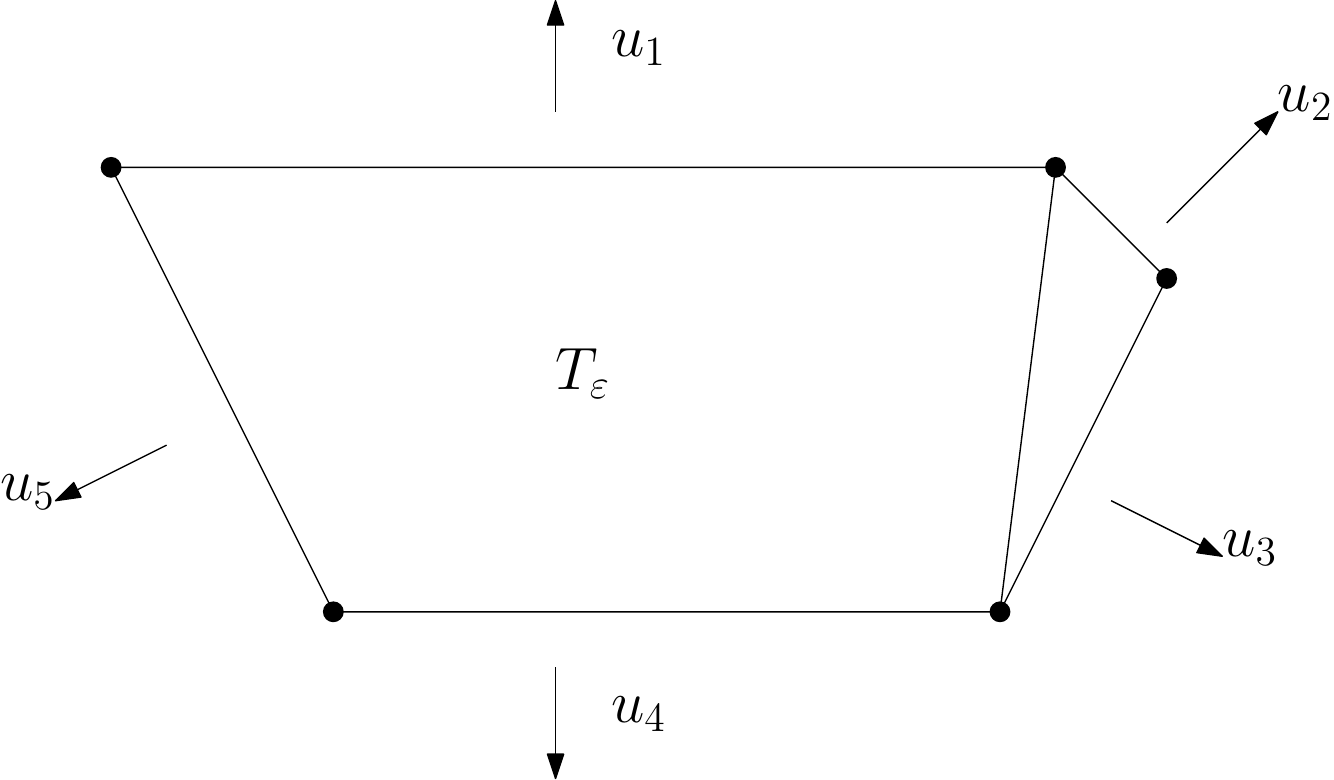}
\caption{} 
\end{figure}
Therefore the only pair of opposite vectors is  $u_1,u_4$.
In  addition, choose $\alpha,\beta\in(0,1)$ such that $\alpha>\beta$ and $\sqrt{\alpha}+\sqrt{\beta}=1$. In particular,
$\alpha+\beta>1/2$.

For all $0< \varepsilon\leq\frac12$, let $\mu_{\varepsilon}$ be the discrete probability measure concentrated on
$u_1,\ldots,u_5$ such that
\begin{eqnarray*}
\mu_{\varepsilon}(\{u_1\})&=&(1-\varepsilon)\alpha\\
\mu_{\varepsilon}(\{u_4\})&=&(1-\varepsilon)\beta\\
\mu_{\varepsilon}(\{u_i\})&=&\mbox{$\frac13$}[1-(1-\varepsilon)(\alpha+\beta)]\mbox{ \ for $i=2,3,5$}.
\end{eqnarray*}
According to Theorem~B, $\mu_{\varepsilon}$ is a cone volume measure if $\varepsilon$ is relatively large (more precisely, if $(1-\varepsilon)(\alpha+\beta)<\frac12$).

\begin{lemma}
\label{nopolygon}
There exists $\varepsilon_0>0$ such that $\mu_{\varepsilon}$ is not a cone volume measure
for $\varepsilon\in(0,\varepsilon_0)$.
\end{lemma}
\begin{proof} Let us suppose that for some sequence $\varepsilon_m>0$ tending to zero,
 there exists a polygon $Q_{\varepsilon_m}$ whose cone volume measure
is $\mu_{\varepsilon_m}$, and seek a contradiction.

For $\varepsilon\in\{\varepsilon_m\}$, let $l_{i,\varepsilon}$ be the length of $F(Q_\varepsilon,u_i)$, and let
$T_\varepsilon=[F(Q_\varepsilon,u_1),F(Q_\varepsilon,u_4)]$ be the trapezoid determined by the two parallel sides of $Q_\varepsilon$ (see Figure 1). Since 
$$
V(T_\varepsilon)\geq (\sqrt{(1-\varepsilon)\alpha}+\sqrt{(1-\varepsilon)\beta})^2=1-\varepsilon
$$
according to Corollary~\ref{oppositefacets2}, we deduce that
\begin{equation}
\label{restarea}
V(Q_\varepsilon\backslash T_\varepsilon)\leq \varepsilon.
\end{equation}

We claim that there exists $\omega>0$ such that if $\varepsilon\in\{\varepsilon_m\}$ is close to zero, then
\begin{equation}
\label{plussconditionl}
l_{1,\varepsilon}>(1+\omega)l_{4,\varepsilon}.
\end{equation}
We define $\varrho_\varepsilon=\sqrt{l_{1,\varepsilon}/l_{4,\varepsilon}}$, and
hence the proof of Theorem~\ref{oppositefacets} shows that
$$
V(T_\varepsilon)=f_{(1-\varepsilon)\alpha,(1-\varepsilon)\beta,2}(\varrho_\varepsilon).
$$ 
It follows from  (\ref{phi2}) that
\begin{align*}
f_{\alpha,\beta,2}(\varrho_\varepsilon)=&(1-\varepsilon)^{-1}f_{(1-\varepsilon)\alpha,(1-\varepsilon)\beta,2}(\varrho_\varepsilon) =(1-\varepsilon)^{-1}V(T_\varepsilon)\\
\leq&(1-\varepsilon)^{-1}=(1-\varepsilon)^{-1}(\sqrt{\alpha}+\sqrt{\beta})^2= (1-\varepsilon)^{-1}\varphi(\alpha,\beta,2).
\end{align*}
Since $\varrho=\sqrt[4]{\alpha/\beta}>1$ is the unique minimum point of the convex function 
$f_{\alpha,\beta,2}(\varrho)$ according to (\ref{phi2}), we deduce that 
$\varrho_\varepsilon>\sqrt[8]{\alpha/\beta}$ if $\varepsilon$ is small.
In particular, (\ref{plussconditionl}) follows as 
$l_{1,\varepsilon}/l_{4,\varepsilon}=\varrho_\varepsilon^2>1+\omega$
for $\omega=\sqrt[4]{\alpha/\beta}-1$.

Possibly decreasing $\omega>0$, we may also assume that in addition to (\ref{plussconditionl}), we have
\begin{eqnarray}
\label{plusscondition}
\mu_{\varepsilon}(\{u_i\})>&\omega &\mbox{ \ for $i=1,\ldots,5$,}\\
\label{plusscondition0}
\langle u_i,u_j\rangle>&-1+\omega &\mbox{ \  for $i,j=1,\ldots,5$, $\{i,j\}\neq \{1,4\}$.}
\end{eqnarray}

According to the Blaschke selection theorem, we may assume that
either  the diameter 
${\rm diam}Q_{\varepsilon_m}$ of $Q_{\varepsilon_m}$ tends to infinity, or $\lim_{m\to\infty}Q_{\varepsilon_m}=T$ for a polygon $T$.\\

\noindent{\bf Case 1 } $\lim_{m\to\infty}{\rm diam}Q_{\varepsilon_m}=\infty$\\
Let $v$ be a unit normal orthogonal to $u_1$, and let
$$
w_m=h_{Q_{\varepsilon_m}}(u_1)+h_{Q_{\varepsilon_m}}(u_4)
$$
be the width of $Q_{\varepsilon_m}$ in the direction of $u_1$.
We claim that
\begin{equation}
\label{wmsmall}
\lim_{m\to\infty}w_m=0.
\end{equation}
To prove (\ref{wmsmall}), let $D_m$ be a circular disc of largest radius inscribed into $Q_{\varepsilon_m}$, and let $r_m$ be the radius of $D_m$. In particular, 
\begin{description}
\item[(a)] either there exist three sides of $Q_{\varepsilon_m}$ such that their lines determine a triangle  whose inscribed circular disc is $D_m$,
\item[(b)] or there exist parallel sides of  $Q_{\varepsilon_m}$ whose lines are of distance $2r_m$.
\end{description}
Since $V(Q_{\varepsilon_m})=1$ and $\lim_{m\to\infty}{\rm diam}Q_{\varepsilon_m}=\infty$, we have
$\lim_{m\to\infty}r_m=0$. If (a) holds,  then the exterior unit normals of the three sides are among $u_1,\ldots,u_5$, and hence (\ref{plusscondition0}) yields a constant $c>0$ depending on $u_1,\ldots,u_5$ such that 
${\rm diam}Q_{\varepsilon_m}\leq cr_m$.
It follows from $\lim_{m\to\infty}{\rm diam}Q_{\varepsilon_m}=\infty$ that (b) holds for large $m$, which in turn implies (\ref{wmsmall}). Naturally, the two parallel sides are
$F(Q_\varepsilon,u_1)$ and $F(Q_\varepsilon,u_4)$.

 Since $|\langle v,u_i\rangle|\geq \frac{\sqrt{3}}2$ for $i=2,3,5$, we deduce that
$$
l_{i,\varepsilon_m}<\mbox{$\frac2{\sqrt{3}}w_m$ for $i=2,3,5$},
$$
which in turn yields {\it via} the triangle inequality that
\begin{equation}
\label{l1l4diff}
|l_{1,\varepsilon_m}-l_{4,\varepsilon_m}|<\mbox{$\frac6{\sqrt{3}}w_m$}.
\end{equation}
However, $V(T_{\varepsilon_m})=\frac{w_m}2(l_{1,\varepsilon_m}+l_{4,\varepsilon_m})$ and
$1-\varepsilon_m\leq V(T_{\varepsilon_m})<1$, which combined with (\ref{wmsmall}) and (\ref{l1l4diff}) imply
$$
\lim_{m\to\infty}\frac{l_{1,\varepsilon_m}}{l_{4,\varepsilon_m}}=1.
$$
The last formula contradicts (\ref{plussconditionl}). \\

\noindent{\bf Case 2 } $\lim_{m\to\infty}Q_{\varepsilon_m}=T$ for a polygon $T$\\
Since $\lim_{m\to\infty}V(Q_{\varepsilon_m}\backslash T_{\varepsilon_m})=0$ by (\ref{restarea}), we have
$$
T=\lim_{m\to\infty}Q_{\varepsilon_m}=\lim_{m\to\infty}T_{\varepsilon_m}.
$$
Therefore $T$ is  a trapezoid. It follows that there is a $u_j$ not contained in ${\rm supp}V_T$, and let $f$ be a continuous function on $S^1$ satisfying $f(u_j)=1$ and $f(u)=0$ for $u\in {\rm supp}V_T$. As $V_{Q_{\varepsilon_m}}$ tends weakly to $V_T$ by the weak continuity of the cone volume measure (see R. Schneider \cite{Sch93}),  we conclude from (\ref{plusscondition}) that
$$
0=\int_{S^1}f\,dV_T=\lim_{m\to\infty}\int_{S^1}f\,dV_{Q_{\varepsilon_m}}\geq \omega.
$$
This contradiction finally verifies Lemma~\ref{nopolygon}.
\end{proof}

\section{Estimates in higher dimensions}
\label{sec3}

For $n\geq 3$, we do not expect a close formula for the expressions $\varphi(\alpha,\beta,n)$ and $\varrho_0(\gamma,n)$ playing crucial roles in Theorem~\ref{oppositefacets}. Therefore we collect some additional properties besides (\ref{philow}).
In this section, the implied constant in $O(\cdot)$ depends on $n$. First we prove the estimates for 
$\varrho_0(\frac{\alpha}{\beta},n)$.\\

\noindent{\bf Proof of Theorem~\ref{rho0}: } Let us recall that $\gamma>1$. We breviate $f_{\gamma,1,n}(\varrho)$ (see (\ref{fdef})) as $f_\gamma(\varrho)$, and hence $\varrho_0(\gamma,n)$ is a root of the equation
\begin{equation}
0=f'_\gamma(\varrho)=\frac{2}{\varrho}\cdot
\sum_{i=1}^{n-1}i\left(\varrho^{2i}-\frac{\gamma}{\varrho^{2i}}\right).
\end{equation}
We see that the sign of $f'_\gamma(\varrho)$ is that of
\begin{equation}
\label{fder}
h(\varrho)=\sum_{i=1}^{n-1}i\left(\frac{\varrho^{2i}}{\gamma^{1/2}}-\frac{\gamma^{1/2}}{\varrho^{2i}}\right).
\end{equation}
Since $h(\varrho)$ is strictly increasing for $\varrho\geq 1$, it is negative for $\varrho=1$ and tends to infinity as $\varrho$ tends to infinity, $\varrho_0(\gamma,n)$ is the unique root of (\ref{fder}). Let
$$
\mbox{$\tau(\gamma)=\frac{\log \varrho_0(\gamma,n)}{\log \gamma}$, and hence $\varrho_0(\gamma,n)=\gamma^{\tau(\gamma)}$}.
$$

As general estimates, we prove that if $\gamma>1$, then
\begin{equation}
\label{taulowupp}
\mbox{$\frac{3}{4(2n-1)}<\tau(\gamma)<\frac1{2n}$}.
\end{equation}
First, we observe that if $\varrho=\gamma^{\frac1{2n}}$, and $1\leq j<n/2<i$ such that $2i-n=n-2j$, then for
$$
t=\frac{\varrho^{2i}}{\gamma^{\frac12}}=\frac{\gamma^{\frac12}}{\varrho^{2j}}>1,
$$
we have
\begin{equation}
\label{ij}
j\left(\frac{\varrho^{2j}}{\gamma^{1/2}}-\frac{\gamma^{1/2}}{\varrho^{2j}}\right)+
i\left(\frac{\varrho^{2i}}{\gamma^{1/2}}-\frac{\gamma^{1/2}}{\varrho^{2i}}\right)=
\left(\frac{j}{t}-jt+it-\frac{i}{t}\right)=\frac{(i-j)(t^2-1)}{t}>0.
\end{equation}
Summing (\ref{ij}) for all pairs $\{j=n-i,i\}$ for $i=\lceil \frac{n+1}2\rceil,\ldots,n-1$ yields that
$h(\varrho)>0$ for $\varrho=\gamma^{\frac1{2n}}$.
As $h(\varrho)$ is strictly increasing, we conclude the upper bound in (\ref{taulowupp}).

Establishing the lower bound is more complicated. Put $\tau(\gamma)=\frac{3}{4(2n-1)}$ and $\varrho=\gamma^{\frac{3}{4(2n-1)}}$. We prove that $h(\varrho)<0$ in this case, it implies the lower bound as $h$ is increasing.

Define $s(x)=\gamma^{x}-\gamma^{-x}.$ It is a strictly increasing odd function, convex for $x\in[0,\infty)$. Hence we have the following two properties:
\begin{enumerate}
  \item[(*)] $s(a)+s(b)>s(a+y)+s(b-y)$ for $0\le a<a+y\le b-y<b$;
  \item[(**)] $s(b)/b< s(a)/a$ in other words $s(b)< \frac{b}{a}s(a)$ for $0<a<b$.
\end{enumerate}

Using this functional notation, we may rewrite
\begin{equation}\label{sumis}
h(\varrho)=\sum_{i=1}^{n-1}i s(2i\tau-1/2)=\sum_{i=1}^{n-1}i s\left(\frac{3i-(2n-1)}{4n-2}\right).
\end{equation}

Why this particular $\tau$ is the cut-off point is illustrated by the fact that the mean value of the $\frac{n(n-1)}{2}$ arguments in the above sum is
\[\frac{\frac{1}{4n-2}\left(3\frac{(n-1)n(2n-1)}{6}-(2n-1)\frac{n(n-1)}{2}\right)}{\frac{n(n-1)}{2}}=0.
\]

We split the summands in \eqref{sumis} into three groups $\mathcal{A,B,C}$. For $s(x)$ in the sum let $x\in\mathcal{A}\Leftrightarrow x< \frac{2-n}{4n-2}$, $x\in \mathcal{B}\Leftrightarrow \frac{2-n}{4n-2}\leq x\le 0$ and $x\in \mathcal{C}\Leftrightarrow 0<x$. Consequently the index (and multiplicity) $i$ in the sum \eqref{sumis} defines a summand in the middle group if $\frac{n+1}{3}\leq i\le\frac{2n-1}{3}$.

We shall prove that $h(\varrho)$ is negative, using the inequalities (*), (**) above. First we use (*) to cancel the negative terms $s(x)$ for $x\in\mathcal B$ thus leaving only negative arguments with absolute value at least $\frac{n-2}{4n-2}$. The positive arguments however are all at most $\frac{n-2}{4n-2}$, hence  we can use (**) to change every summand into $s(\frac{n-2}{4n-2})$, and it will have coefficient $0$.

We split the proof into three cases depending on the residue of $n$ mod $3$. The case when $3\mid 2n-1$ is the simplest and illustrates the two other cases. Here $x\in\mathcal{B}$ if $x(4n-2)\in\{-n+2,-n+5,\ldots,-3,0\}$ and $x\in\mathcal{B}\Rightarrow -x\in\mathcal{C}$. In the sum \eqref{sumis} each of the summands $s(x)\;(x\in\mathcal{B})$ occurs with smaller multiplicity than $s(-x)=-s(x)$. So after canceling them for every negative term $s(x)$ left in the sum we have $x\le\frac{-n-1}{4n-2}$. In particular, we may write
\begin{equation*}
h(\varrho)=\sum_{i=1}^{\frac{n-2}{3}}i s\left(\frac{3i-(2n-1)}{4n-2}\right)+\sum_{i=\frac{2n+2}{3}}^{n-1}\left(i-\left(\frac{4n-2}{3}-i\right)\right) s\left(\frac{3i-(2n-1)}{4n-2}\right).
\end{equation*}

For each of these negative terms  we use (**) (for $-x$) to get $s(x)\le\frac{x(4n-2)}{n-2} s(\frac{n-2}{4n-2})$. For each of the positive terms $s(y)$ we have $y\leq \frac{n-2}{4n-2}$ so similarly $s(y)\le \frac{y(4n-2)}{n-2} s(\frac{n-2}{4n-2})$.

Neither the cancelations, nor these approximations changed the weighted mean of the arguments, it is still $0$. But now all the arguments are equal, so the sum of the coefficients is $0$. Putting these approximations into \eqref{sumis} we conclude that $h(\varrho)<0$.

Let now $n-1$ be divisible by $3$. Then $x\in\mathcal{B}$ if $x(4n-2)\in\{3-n,\ldots, -1\}$ and $x\in\mathcal(B)\Rightarrow -x+\frac{1}{4n-2}\in\mathcal{C}$. For each $x\in\mathcal{B}$ we pick another negative summand $s(z)$ (for $z\leq \frac{1-n}{4n-2}$) and use the first approximation to get $s(x)+s(z)<s(x-\frac{1}{4n-2})+s(z+\frac{1}{4n-2})$. With sufficiently many steps we replace each $s(x)$ by $s(x-\frac{1}{4n-2})$ for $x\in\mathcal{B}$. So that these cancel with summands $s(y)$ where $y\in\mathcal{C}$. We can conclude the same way as above if we can make sure that there exists always a suitable $z$ until we replace all summands $s(x)$. The number of such summands in $\mathcal{B}$ is $\frac{n+2}{3}+\frac{n+5}{3}+\cdots+\frac{2n-2}{3}=\frac{n(n-1)}{6}$. The amount by which we can increase (until $\frac{2-n}{4n-2}$) those in $\mathcal{A}$ is
\[\sum_{i=1}^\frac{n-1}{3} i(2n-1-3i-(n-2))=(n+1)\frac{(n-1)(n+2)}{18}-\frac{(n-1)(n+2)(2n+1)}{54}.
\]
Comparing it with $\frac{n(n-1)}{6}$ we get
\[\frac{n(n-1)}{6}\le (n+1)\frac{(n-1)(n+2}{18}-\frac{(n-1)(n+2)(2n+1)}{54}=(n-1)\frac{(n+2)^2}{54}.
\]
This inequality can be written in the form
\[0\le (n+2)^2-9n=n^2-5n+4=(n-4)(n-1),
\]
which holds for $n\ge 4$ as required.

For the last case let $3$ divide $n$. Then $x\in\mathcal{B}$ if $x(4n-2)\in\{4-n,\ldots, -2\}$ and $x\in\mathcal(B)\Rightarrow -x+\frac{2}{4n-2}\in\mathcal{C}$. The multiplicity of $s(\frac{-2}{4n-2})$ is $\frac{2n-3}{3}$ and that of $s(\frac{1}{4n-2})$ is $\frac{2n}{3}$. First we cancel $\frac{n}{3}$ of these by \[\frac{n}{3}s\left(\frac{-2}{4n-2}\right)+\frac{2n}{3}s\left(\frac{1}{4n-2}\right)<
\frac{2n}{3}\left(s\left(\frac{-1}{4n-2}\right)+s\left(\frac{1}{4n-2}\right)\right)=0.\] For $n=3$ there is nothing more to do. So assume $n\geq 6$. For each of the remaining $s(x)$ ($x\in\mathcal{B}$) we pick two other (possibly equal) negative summands $s(z_1),\,s(z_2)$ (for $z_1,\,z_2\in \mathcal{A}$) and use (**) twice to get \[s(x)+s(z_1)+s(z_2)<s\left(x-\frac{2}{4n-2}\right)+s\left(z_1+\frac{1}{4n-2}\right)+s\left(z_2+\frac{2}{4n-2}\right).\] With sufficiently many steps we replace each $s(x)$ by $s(x-\frac{2}{4n-2})$ for $x\in\mathcal{B}$. So that these cancel with summands $s(y)$ where $y\in\mathcal{C}$. We can conclude the same way as above if we can make sure that there exists always a suitable $z$ until we replace all summands $s(x)$. The number of such summands is $\frac{n+3}{3}+\frac{n+6}{3}+\cdots+\frac{2n-3}{3}-\frac{n}{3}=\frac{n(n-5)}{6}$. The amount by which we can decrease (until $2-n$) those in $\mathcal{A}$ is
\[\sum_{i=1}^\frac{n}{3} i(n+1-3i)=(n+1)\frac{n(n+3)}{18}-\frac{n(n+3)(2n+3)}{54}.
\]

Comparing it with $2\frac{n(n-5)}{6}$ we get
\[\frac{n(n-5)}{3}\le \frac{(n+3)n^2}{54}.
\]
This inequality can be written in the form
\[0\le n^2-15n+90,
\]
which always holds. With this we have finished the proof of  (\ref{taulowupp}).

To have a polynomial equation for $\varrho_0(\gamma,n)$, we use the formula
\begin{equation}
\label{ysum}
\sum_{i=1}^{n-1}iy^i=\frac{y}{(y-1)^2}\,((n-1)y^n-ny^{n-1}+1)
\end{equation}
for $y\neq 1$, which can be obtained by derivating $\frac{1-y^n}{1-y}$ for $y>1$. It follows from (\ref{fder}), (\ref{ysum}) and $\frac{y}{(y-1)^2}=\frac{1/y}{((1/y)-1)^2}$ that
if $\varrho=\varrho_0(\gamma,n)$, then
\begin{align}
\nonumber
0&=\frac{\varrho^2}{(\varrho^2-1)^2}\,\left[(n-1)\varrho^{2n}-n\varrho^{2(n-1)}+1\right]-
\frac{\gamma\cdot\varrho^{-2}}{(\varrho^{-2}-1)^2}\,\left[(n-1)\varrho^{-2n}-n\varrho^{-2(n-1)}+1\right]\\
\label{rhopolynomial}
&=\frac{\gamma^{\frac12}\varrho^2}{(\varrho^2-1)^2}\,\left[
(n-1)s(2n\tau-1/2)-ns(2(n-1)\tau-1/2)-s(1/2)   \right].
\end{align}

First we assume that $\gamma>1$ tends to one, and hence $\gamma=1+\varepsilon$ where  $\varepsilon>0$ tends to zero.
For fixed $r\in [-1,1]$ and  small $\varepsilon>0$, the binomial series implies
\begin{equation}
\label{binom}
s(r)=(1+\varepsilon)^r-(1+\varepsilon)^{-r}=
r(2\varepsilon-\varepsilon^2+\mbox{$\frac23$}\,\varepsilon^3)+r^3\cdot\mbox{$\frac23$}\,\varepsilon^3 +O(\varepsilon^4).
 \end{equation}
We apply (\ref{binom}) for  $r=2n\tau-\frac12$, $r=2(n-1)\tau-\frac12$ and $r=-\frac12$, respectively.
Since
$$
(n-1)\left(2n\tau-\frac12\right)-n\left(2(n-1)\tau-\frac12 \right)-\frac12=0,
$$
and
\begin{align*}
(n-1)\left(2n\tau-\frac12\right)^3-n\left(2(n-1)\tau-\frac12 \right)^3-\left(\frac12\right)^3&=\\
\tau^2(6(n-1)n^2-6n(n-1)^2)-\tau^3(8(n-1)n^3-8n(n-1)^3))&=\\
2n(n-1)\tau^2\left(3-\tau(8n-4)\right),&
\end{align*}
we deduce by (\ref{binom})  that
$$
2n(n-1)\tau^2\left(3-\tau(8n-4)\right)\cdot\mbox{$\frac23$}\,\varepsilon^3 =O(\varepsilon^4).
$$
In turn, (\ref{taulowupp}) yields that if $\varepsilon>0$ is small, then
\begin{equation}
\label{gamma1tau}
\tau(1+\varepsilon)=\frac{3}{8n-4}+O(\varepsilon).
\end{equation}

Next let $\gamma$ tend to infinity. It follows from (\ref{rhopolynomial}) that for $\varrho=\varrho_0(\gamma,n)$, we have
$$
(n-1)\varrho^{2n}-n\varrho^{2(n-1)}=\gamma+(n-1)\gamma\varrho^{-2n}-n\gamma\varrho^{-2(n-1)}-1.
$$
As $\varrho>\gamma^{\frac1{4n}}$ by (\ref{taulowupp}), we deduce that
\begin{equation}
\label{rhogammainfty}
\varrho_0(\gamma,n)=(n-1)^{\frac{-1}{2n}}\gamma^{\frac1{2n}}\left(1+O\left(\varrho_0(\gamma,n)^{-2}\right) \right)
=(n-1)^{\frac{-1}{2n}}\gamma^{\frac1{2n}}\left(1+O\left(\gamma^{-\frac1{2(n-1)}}\right) \right).
\end{equation}
In particular, $\lim_{\gamma\to \infty}\tau(\gamma)=\frac1{2n}$ and $\lim_{\gamma\to 1^+}\tau(\gamma)=\frac3{8n-4}$.
\proofbox

Next prove the bounds for $\varphi(\alpha,\beta,n)$ if $n\geq 3$.\\

\noindent{\bf Proof of Corollary~\ref{phinbig}: }
We use the notation of the proof of  Theorem~\ref{rho0}. The lower bound in (i) is just (\ref{philow}).

For the upper bound in (i), we may asume that $\alpha>\beta$. We have
\begin{equation}
\label{fform}
f_{\alpha,\beta,n}(\varrho)=\alpha+\beta+\frac{\alpha(1-\varrho^{-2n+2})+\beta(\varrho^{2n}-\varrho^2)}{\varrho^2-1}.
\end{equation}
Defining $\varrho>0$ by $\varrho^{2n}=\alpha/\beta$, the upper  bound in (i) is a consequence of
$$
\frac{\alpha(1-\varrho^{-2n+2})+\beta(\varrho^{2n}-\varrho^2)}{\varrho^2-1}< 2(n-1)\alpha^{\frac{n-1}n}\beta^{\frac1n}.
$$
Dividing by $\beta$,  it is sufficient to prove that if $\varrho>1$, then
\begin{equation}
\label{uppfirst}
\varrho^{2n}(1-\varrho^{-2n+2})+\varrho^{2n}-\varrho^2< 2(n-1)\varrho^{2n-2}(\varrho^2-1),
\end{equation}
which is  equivalent with
$$
(n-1)\varrho^{2n-2}<(n-2)\varrho^{2n}+\varrho^2.
$$
Since the last inequality holds for $\varrho>1$ by the inequality between arithmetic and geometric mean, we conclude (\ref{uppfirst}), and in turn (i).

If $\frac{\alpha}{\beta}=1+\varepsilon$ for small $\varepsilon>0$, then for $r\in(0,1)$, we have the formula
$$
(1+\varepsilon)^r+(1+\varepsilon)^{-r}=2+r^2\varepsilon^2+O(\varepsilon^3).
$$
 We substitute $\varrho=(1+\varepsilon)^\tau$ for $\tau=\frac{3}{8n-4}+O(\varepsilon)$ given by (\ref{gamma1tau}) into (\ref{fdef}).  For $i=1,\ldots,n-1$, we have
\begin{align*}
\alpha\varrho^{-2i}+\beta\varrho^{2i}&=\sqrt{\alpha\beta}
\left((1+\varepsilon)^{\frac12-2i\tau}+(1+\varepsilon)^{-(\frac12-2i\tau)}\right)=
\sqrt{\alpha\beta}\left(2+(\mbox{$\frac12$}-2i\tau)^2\varepsilon^2+O(\varepsilon^3)\right)\\
&=\sqrt{\alpha\beta}\left(2+(\mbox{$\frac12-\frac{3i}{4n-2}$})^2\varepsilon^2+O(\varepsilon^3)\right).
\end{align*}
Since
$$
\sum_{i=1}^{n-1}\left(\frac12-\frac{3i}{4n-2}\right)^2=\frac{n-1}4-\frac{3n(n-1)}{8n-4}+\frac{9n(n-1)(2n-1)}{6(4n-2)^2}=
\frac{(n-1)(n-2)}{8n-4},
$$
we conclude the formula in (ii).

Let us assume that $\gamma=\frac{\alpha}{\beta}$ is large. We deduce by (\ref{fdef}), (\ref{rhogammainfty})
and (\ref{fform}) that
for $\varrho=\varrho_0(\gamma,n)$, we have
\begin{align*}
f_{\alpha,\beta,n}(\varrho)&
=\alpha+\beta+\frac{\alpha }{\varrho^2}+
\frac{\beta \gamma}{(n-1)\varrho^2}+O\left(\frac{\alpha}{\varrho^4}\right)
=\alpha+\beta+\frac{\alpha n}{(n-1)\varrho^2}+O\left(\frac{\alpha}{\varrho^4}\right)\\
&=\alpha+\beta+\frac{\alpha n}{(n-1)^{\frac{n-1}n}\gamma^{\frac1n}}+O\left(\frac{\alpha}{\gamma^{\frac2n}}\right).
\end{align*}
\proofbox

\noindent{\bf Acknowledgement } We thank Gaungxian Zhu for useful discussions, and correcting various mistakes in an earlier version, and the unknown referee for many improvements. First named author is supported by OTKA 109789, and the second named author is supported by OTKA 84233.


\begin{thebibliography}{99}








\bibitem{Ball:1988vo}
K.M. Ball, \emph{Logarithmically concave functions and sections of convex sets in $\R^{n}$},
Studia Math. \textbf{88} (1988), 69--84.

\bibitem{Bal91}
K.M. Ball,  
\emph{Volume ratios and a reverse isoperimetric inequality},  
J. London Math. Soc. \textbf{44} (1991), 351--359.

\bibitem{Bar98}
F. Barthe, 
\emph{On a reverse form of the Brascamp-Lieb inequality}, 
 Invent. Math. \textbf{134} (1998), 335--361.

\bibitem{Bar04}
F. Barthe,
\emph{A continuous version of the Brascamp-Lieb inequalities},  
in Geometric aspects of functional analysis, Lecture Notes in Math., \textbf{1850}, 53--63, Springer, Berlin, 2004.



\bibitem{BCE13}
F. Barthe, D. Cordero-Erausquin, \emph{Invariances in variance  estimates},  
Proc. Lond. Math. Soc,. \textbf{106} (2013), 33--64.

\bibitem{BGMN05}
F. Barthe, O. Guedon, S. Mendelson, A. Naor, 
\emph{A probabilistic approach to the geometry of the $l^n_p$-ball},  
Ann. of Probability, \textbf{33} (2005), 480--513.


\bibitem{BCCT07}
J. Bennett, A. Carbery, M. Christ, T. Tao, 
\emph{The Brascamp–Lieb inequalities: finiteness, structure, and
  extremals}, Geom. funct. anal., \textbf{17} (2007), 1343--1415.





\bibitem{BHZ}
K.J. B\"or\"oczky, P. Heged\H{u}s, G. Zhu,
\emph{The discrete logarithmic Minkowski problem},
IMRN,  submitted.
arXiv:1409.7907, 4fu9n

\bibitem{BH}
 K.J. B\"{o}r\"{o}czky, M. Henk: \emph{Cone-volume measure and stability}, 
arXiv:1407.7272, submitted.

\bibitem{BLYZ12}
  K.J. B\"or\"oczky, E. Lutwak, D. Yang and  G. Zhang,
 \emph{The log-Brunn-Minkowski inequality},
 Adv. Math., \textbf{231} (2012), 1974--1997.

\bibitem{BLYZ13}
 K.J. B\"or\"oczky, E. Lutwak, D. Yang and  G. Zhang,
\emph{The logarithmic Minkowski problem},
J.~AMS, {\bf 26} (2013), 831--852.

\bibitem{BLYZ14}
 K.J. B\"or\"oczky, E. Lutwak, D. Yang,  G. Zhang,
\emph{Affine images of isotropic measures}, to appear in 
J. Diff. Geom.


\bibitem{CCE09}
E. Carlen, D. Cordero-Erausquin, 
\emph{Subadditivity of the entropy and its relation to Brascamp-Lieb type inequalities}, 
Geom. Funct. Anal., \textbf{19} (2009), 373--405.





























\bibitem{Gar95}
R.J.~Gardner, \emph{Geometric {T}omography},
2nd edition, Encyclopedia of
Mathematics and
 its Applications, vol.~\textbf{58}, Cambridge University Press, Cambridge,
 2006.

\bibitem{Gardnersurvey}
R.J.~Gardner, \emph{The Brunn-Minkowski inequality},
Bull. Amer. Math. Soc.
\textbf{39} (2002), 355--405.



\bibitem{GrM87}
M. Gromov, V.D. Milman, 
\emph{Generalization of the spherical isoperimetric inequality
for uniformly convex Banach Spaces}, 
Composito Math. \textbf{62} (1987), 263--282.


\bibitem{Gruberbook}
P.M.~Gruber, \emph{Convex and discrete geometry},
Grundlehren der
Mathematischen Wissenschaften, \textbf{336}, Springer, Berlin, 2007.


\bibitem{GuM11}
O. Guedon, E. Milman, 
\emph{Interpolating thin-shell and sharp large-deviation estimates for
  isotropic log-concave measures}, 
Geom. Funct. Anal. \textbf{21} (2011), 1043--1068.

\bibitem{HaberlParapatits:2013}
Ch. Haberl and L. Parapatits,  
\emph{The centro-affine Hadwiger theorem},   
to appear in J. Amer. Math. Soc. (JAMS).














\bibitem{HLL06}
B. He, G. Leng, K. Li, 
\emph{Projection problems for symmetric polytopes},  
Adv. Math., \textbf{207} (2006), 73--90.

\bibitem{HSW05}
M. Henk, A. Sch\"urmann, J.M.Wills, 
\emph{Ehrhart polynomials and successive minima}, 
Mathematika, \textbf{52} (2005), 1--16.

\bibitem{HeL14}
M. Henk, E. Linke, 
\emph{Cone-volume measures of polytopes},
Adv. Math. \textbf{253} (2014), 50--62.









 \bibitem{KLM95}
R. Kannan, L. Lov\'asz, M. Simonovits, 
\emph{Isoperimetric problems for convex bodies and a localization
  lemma},  
Discrete Comput. Geom. \textbf{13} (1995), 541--559.

 \bibitem{Kla09}
B. Klartag, 
\emph{A Berry-Esseen type inequality for convex bodies with an unconditional basis},
 Probab. Theory Related Fields \textbf{145} (2009), 1--33.

 \bibitem{Kla10}
B. Klartag, 
\emph{On nearly radial marginals of high-dimensional probability measures},
J. Eur. Math. Soc., \textbf{12} (2010), 723--754.














\bibitem{Lud10}
M. Ludwig, 
\emph{General affine surface areas},
Adv. Math., \textbf{224} (2010), 2346--2360.


\bibitem{LuR10}
M. Ludwig, M. Reitzner, \emph{A classification of $SL(n)$ invariant valuations},
Ann. of Math., \textbf{172} (2010), 1223--1271.





 \bibitem{Lut93b}
E.\ Lutwak, \emph{{The Brunn-Minkowski-Firey theory. I. Mixed volumes
 and the Minkowski problem}},
 J. Differential Geom. \textbf{38}
 (1993), 131--150, MR 1231704, Zbl 0788.52007.











\bibitem{LYZ01}
E. Lutwak, D. Yang, G. Zhang, 
\emph{A new affine invariant for polytopes and Schneider's projection
  problem},  
Trans. Amer. Math. Soc., \textbf{353} (2001), 1767--1779.



\bibitem{LYZ04}
E. Lutwak, D. Yang, G. Zhang, 
\emph{Volume inequalities for subspaces of $L_p$},
J. Differential Geom., \textbf{68} (2004), 159--184.


\bibitem{LYZ05}
E. Lutwak, D. Yang,  G. Zhang, 
\emph{$L^p$ John ellipsoids}, 
Proc. London Math. Soc., \textbf{90} (2005), 497--520.


\bibitem{LYZ07}
E.\ Lutwak, D.\ Yang, G.\ Zhang, 
\emph{Volume inequalities for isotropic measures}, 
Amer. J. Math., \textbf{129} (2007), 1711--1723.







\bibitem{Nao07}
A. Naor, \emph{The surface measure and cone measure on the sphere of $l^n_p$}, 
Trans. Amer. Math. Soc., \textbf{359} (2007), 1045--1079.

\bibitem{NaR03}
A. Naor, D. Romik, 
\emph{Projecting the surface measure of the sphere of $l^n_p$}, 
Ann. Inst. H. Poincar\'e Probab. Statist. \textbf{39} (2003), 241--261.






\bibitem{PaW12}
G.  Paouris, E. Werner, 
\emph{Relative entropy of cone measures and $L_p$ centroid bodies}, 
Proc. London Math. Soc, \textbf{104} (2012), 253--286.



\bibitem{Pfe12}
W.F. Pfeffer, 
\emph{The divergence theorem and sets of finite perimeter},  
 CRC Press, Boca Raton, FL, 2012.





\bibitem{Rockafellar:1997ww}
R.T. Rockafellar, 
\emph{Convex analysis}, Princeton University Press, 1997.





\bibitem{Sch93}
R.~Schneider, \emph{Convex bodies: the {B}runn-{M}inkowski theory},
  Encyclopedia of Mathematics and its Applications, vol.~\textbf{44}, Cambridge
  University Press, Cambridge, 1993, Second expanded edition, 2014.











\bibitem{Sta02}
A.~Stancu,  
\emph{The discrete planar $L_0$-Minkowski problem}, 
Adv. Math., \textbf{167} (2002), 160--174.

 \bibitem{Sta03}
 A. Stancu, 
\emph{On the number of solutions to the discrete two-dimensional
  $L_0$-Minkowski problem}, 
  Adv. Math. \textbf{180} (2003), 290-323.

\bibitem{Sta12}
A. Stancu, 
\emph{Centro-affine invariants for smooth convex bodies},  
Int. Math. Res. Not., (2012), 2289--2320.

\bibitem{Tho96}
A.~C. Thompson, \emph{Minkowski geometry}, Encyclopedia of
Mathematics and its Applications, vol.~63, Cambridge University
Press, Cambridge, 1996.









\bibitem{Xio10}
G. Xiong, 
\emph{Extremum problems for the cone-volume functional of convex
  polytopes}, 
Adv. Math. \textbf{225} (2010), 3214--3228.





\bibitem{Zhu14a}
G.~Zhu,
\emph{The logarithmic Minkowski problem for polytopes},
Adv. Math. \textbf{262} (2014), 909--931.

\bibitem{Zhu14b}
G.~Zhu,
\emph{The centro-affine Minkowski problem for polytopes},
accepted for publication in J Differential Geom.

 \bibitem{Zhu14c}
 G.~Zhu,
 \emph{The $L_p$ Minkowski problem for polytopes},
 arXiv:1406.7503.



\end{thebibliography}
\end{document}